\documentclass{amsart}
\usepackage{amssymb,amstext,amsmath,amscd,amsthm,amsfonts,enumerate,graphicx,latexsym}
\usepackage[usenames]{color}

\newtheorem{theorem}{Theorem}[section]
\newtheorem{lemma}[theorem]{Lemma}
\newtheorem{corollary}[theorem]{Corollary}
\newtheorem{proposition}[theorem]{Proposition}
\theoremstyle{definition}
\newtheorem{definition}[theorem]{Definition}

\newtheorem{example}[theorem]{Example}
\newtheorem{remark}[theorem]{Remark}

\numberwithin{equation}{section}
\def\A{\mathrm{A}}
\def\CC{\mathsf{K}} 
\def\CM{{\mathsf{CM}}}
\def\cm{\mathsf{K}_\mathrm{CM}}
\def\CMR{\mathsf{CM}(R)}
\def\Coker{\mathrm{Coker}}
\def\D{\mathrm{D}}
\def\dega{\Rightarrow_{\mathrm{deg}}}

\def\e{\mathrm{e}}

\def\m{\mathfrak{m}}
\def\md{\mathsf{M}(d)}

\def\modR{\mathsf{mod}(R)}
\def\OO{\mathcal{O}} 
\def\Rss{R^{\sharp \sharp }}
\def\syz{\Omega}
\def\XX{\mathcal{X}}
\def\YY{\mathcal{Y}}
\def\X{\mathsf{E}}
\begin{document}
\allowdisplaybreaks
\title[A topology on the set of isomorphism classes of maximal Cohen--Macaulay modules]{A topology on the set of isomorphism classes\\of maximal Cohen--Macaulay modules}
\author{Naoya Hiramatsu} 
\address[N. Hiramatsu]{Department of general education, National Institute of Technology, Kure College, 2-2-11, Agaminami, Kure Hiroshima, 737-8506 Japan}
\email{hiramatsu@kure-nct.ac.jp}
\author{Ryo Takahashi}
\address[R. Takahashi]{Graduate School of Mathematics, Nagoya University, Furocho, Chikusaku, Nagoya, Aichi 464-8602, Japan/Department of Mathematics, University of Kansas, Lawrence, KS 66045-7523, USA}
\email{takahashi@math.nagoya-u.ac.jp}
\urladdr{http://www.math.nagoya-u.ac.jp/~takahashi/}
\thanks{NH was supported by JSPS KAKENHI Grant Number 18K13399.
RT was partly supported by JSPS Grant-in-Aid for Scientific Research 16K05098 and JSPS Fund for the Promotion of Joint International Research 16KK0099.}
\begin{abstract}
In this paper, we introduce a topology on the set of isomorphism classes of finitely generated modules over an associative algebra.
Then we focus on the relative topology on the set of isomorphism classes of maximal Cohen--Macaulay modules over a Cohen--Macaulay local ring. 
We discuss the irreducible components over certain hypersurfaces.
\end{abstract}
\maketitle
\section{Introduction}

Let $k$ be an algebraically closed field and $R$ a finite-dimensional $k$-algebra.
Denote by $\md$ the set of $R$-module structures on $k^d$.
Then $\md$ is an algebraic set, which is called the {\em module variety} of $d$-dimensional $R$-modules. 
One of the basic problems is to compute the irreducible components of $\md$, and have been studied by many authors including Gabriel \cite{G74}, Morrison \cite{M80} and Crawley-Boevey and Schr\"oer \cite{CS02}.
Another approach to the structure of $\md$ is made by degenerations.
For $M,N\in\md$ we say that {\em $M$ degenerates to $N$} if $N \in \overline{\mathcal{O}_M}$ in $\md$, where ${\mathcal{O}_M}$ stands for the $\mathrm{GL} _d (k)$-orbit of $M$.
A lot of studies on degenerations have been done, which include Riedtmann \cite{R86}, Zwara \cite{Z00} and so on.

Yoshino \cite{Y} extends the notion of degeneration to arbitrary finitely generated modules over arbitrary associative algebras.
However, as a matter of course, for this extended version of degeneration we no longer have module varieties to develop topological arguments on degeneration.
Thus, in this paper, we introduce a topology on the set of isomorphism classes of maximal Cohen--Macaulay (abbr. MCM) modules over a Cohen--Macaulay (abbr. CM) local ring $R$ by means of degenerations of modules, so that we can regard the set $\X (d)$ of isomorphism classes of MCM $R$-modules of multiplicity $d$ as a substitute for the module variety $\md$ in the case of a finite-dimensional algebra.
We investigate in this paper the irreducible components of $\X(d)$ for hypersurfaces of countable CM representation type.

From now on, we give more precise explanation of our results.

\begin{theorem}[Theorem \ref{A4}]
Let $k$ be a field, and let $R$ be an associative $k$-algebra.
Denote by $\modR$ the set of isomorphism classes of finitely generated left $R$-modules.
For a subset $\XX$ of $\modR$, put
$$
\CC(\XX) = \{ Y \in \modR \mid\text{$X^{\oplus n}$ degenerates to $Y^{\oplus n}$ for some $X \in \XX$ and $n>0$}\}.
$$
Then $\CC(-)$ is a Kuratowski closure operator on $\modR$.
In particular, it induces a topology on $\modR$, so that $\X(d)$ is equipped with a relative topology for each integer $d>0$ in the case where $R$ is a CM local ring and $k$ is its coefficient field.
\end{theorem}

Recall that a CM local ring $R$ is said to have {\em countable CM representation type} if there exist infinitely but only countably many isomorphism classes of indecomposable MCM $R$-modules.
Let $R$ be a complete equicharacteristic local hypersurface with uncountable algebraically closed residue field $k$ of characteristic not two.
Then $R$ has countable CM representation type if and only if $R$ is isomorphic to the ring $k[\![x_0 , x_1 , x_2 , \dots, x_n]\!]/(f)$, where
$$
f=\begin{cases}
x_1^2+x_2^2+\cdots+x_n^2 & (\A_\infty),\\
x_0^2x_1+x_2^2+\cdots+x_n^2 & (\D_\infty).
\end{cases}
$$

\begin{theorem}[Theorem \ref{main}]
Let $k$ be an uncountable algebraically closed field of characteristic not two, and let $R$ be a complete local hypersurface over $k$ having countable CM represenation type.
Then for each integer $d>0$ the topological space $\X(d)$ can be represented as a finite union of irreducible closed subsets, provided that {\rm(i)} $R$ has type $(\D_{\infty})$ or {\rm(ii)} $\dim R$ is odd and $R$ has type $(\A_{\infty})$.
\end{theorem}

The paper is organized as follows.
In Section \ref{Closure}, we recall the notion of degenerations of modules over arbitrary algebras, and introduce a closure operator using degenerations (Definition \ref{A3}), which induces a topology the set of isomorphism classes of MCM modules (Theorem \ref{A4}). 
In Section \ref{Irreducible}, we give a description of the irreducible components of the set of isomorphism classes of MCM modules (Corollary \ref{B2} and Theorem \ref{main}).
\section{A closure operator on $\modR$}\label{Closure}

First of all, let us recall the definition of degenerations of modules over an arbitrary algebra.
Throughout this paper, unless otherwise specified, let $k$ be a field, and let $R$ be an (associative) $k$-algebra.
Denote by $\modR$ the set of isomorphism classes of finitely generated left $R$-modules.

\begin{definition}\cite[Definition 2.1]{Y04}\label{A1}
Let $V = k[t]_{(t)}$ be a localization of a polynomial ring, and let $K= k(t)$ be a rational function field, which is the quotient field of $V$.
For $M,N \in \modR$ we say that {\em $M$ degenerates to $N$} and write $M \dega N$, if there exists a finitely generated left $R\otimes _{k} V$-module $Q$ such that $Q$ is flat as a $V$-module, $Q/tQ \cong N$ as an $R$-module, and $Q_t \cong M\otimes _{k} K$ as an $R\otimes _{k} K$-module.
\end{definition}

We state several fundamental properties of degenerations of modules.

\begin{proposition}\label{A2}
\begin{enumerate}[\rm(1)]
\item
One has $M\dega N$ if and only if there exists an exact sequence
$$
0 \to Z \xrightarrow{\left(
\begin{smallmatrix}
\phi \\
\psi
\end{smallmatrix} \right)
} M\oplus Z \to N \to 0
$$
of finitely generated $R$-modules such that $\psi$ is nilpotent.
If $R$ is a CM local ring and $M,N$ are MCM modules, then $Z$ is MCM as well.
\item
If there exists an exact sequence $0 \to L \to M \to N \to 0$ of finitely generated $R$-modules, then one has $M \dega L \oplus N$.
\item
If $M \dega N$ and $M' \dega N'$, then $M \oplus M' \dega N \oplus N'$.
\item
One has $M\dega M$.
\item
Suppose that $R$ is a CM local ring.
If $M\dega N$ and $N$ is MCM, then $M$ is MCM as well.
\item
If $M \dega N$, then $M$ and $N$ give the same class in the Grothendieck group of $R$.
Hence, if $R$ is a CM local ring and $M,N$ are MCM modules, then $\e(M) = \e(N)$, where $\e(-)$ stands for the (Hilbert--Samuel) multiplicity.
\item
Suppose $M\dega N$.
Then $M=0$ if and only if $N=0$.
\end{enumerate}
\end{proposition}

\begin{proof}
The two assertions in (1) follow from \cite[Theorem 2.2]{Y04} (see also \cite[Theorem 1]{Z00}) and \cite[Remark 4.3]{Y04}, respectively.
The assertion (2) is shown in \cite[Remark 2.5]{Y04}.
The assertions (3) and (6) are easy consequences of (1).
Applying (1) to the trivial exact sequence $0\to M\to M\oplus M\to M\to0$ implies (4).
The assertion (5) follows from \cite[Theorem 3.2]{Y04} and \cite[Corollary 4.7]{Y02}.
Finally, let us show (7).
There is an exact sequence as in (1).
If $M=0$, then the map $\psi$ is injective, and its nilpotency implies $Z=0$, which implies $N=0$.
Conversely, assume $N=0$.
Then there exists a map $\left(
\begin{smallmatrix}
\alpha&\beta\end{smallmatrix} \right):M\oplus Z\to Z$ such that $\left(
\begin{smallmatrix}
\phi \\
\psi
\end{smallmatrix} \right)\left(
\begin{smallmatrix}
\alpha&\beta\end{smallmatrix} \right)=1$.
Then we have $\psi\beta=1$, which implies that $\psi$ is surjective.
Since $\psi$ is nilpotent, the zero map of $Z$ is surjective, which means $Z=0$.
Hence $M=0$.
\end{proof}

For each subset $\XX$ of $\modR$, put
$$
\CC^0 (\XX) = \{ N \in \modR \mid M\dega N\text{ for some }M\in\XX\},
$$
and for each $M\in\modR$ put $\CC^0(M)=\CC^0(\{M\})$.

Suppose that $R$ is a finite-dimensional $k$-algebra.
Then every finitely generated $R$-module is a finite-dimensional $k$-vector space, and for all $N\in\modR$ one has
$$
N\in\CC^0(M)\iff N\in\overline{\OO_M}.
$$
Thus we may regard $\CC^0 (M)$ as a substitute for $\overline{\OO_M}$. 
Moreover, it is clear that
\begin{equation}\label{2}
\CC^0(\CC^0(\XX))=\CC^0(\XX),
\end{equation}
or in other words, $\dega$ is transitive.
(In fact, $\dega$ defines a partial order on $\modR$; see \cite{Z98}.)

On the other hand, in the general (i.e. infinite-dimensional) case, it is unknown whether $\dega$ is transitive, and hence we do not know if \eqref{2} holds.
Recently, this problem has been partially resolved by Takahashi \cite[Theorem 1.2]{T}.

\begin{theorem}\label{T}
Let $R$ be a $k$-algebra and let $L$, $M$, $N \in \modR$. 
Assume that $L \dega M$ and $M \dega N$. 
Then $L^{\oplus n} \dega N^{\oplus n}$ for some integer $n>0$. 
\end{theorem}

Hence, the relation $\dega$ is transitive up to direct sums of copies. 
Taking this into account, we make the following definition.

\begin{definition}\label{A3}
For a subset $\XX$ of $\modR$ we put 
$$
\CC(\XX) = \{ N \in \modR \mid M^{\oplus n} \dega N^{\oplus n} \text{ for some } M \in \XX \text{ and } n>0\}. 
$$
If $\XX$ consists of a single module $M$, then we simply denote it by $\CC(M)$.
\end{definition}

Taking advantage of Theorem \ref{T}, one can prove the following statement.

\begin{theorem}\label{A4}
The assignment $\XX\mapsto\CC(\XX)$ induces a {\em Kuratowski closure operator} on $\modR$, that is,

{\rm(1)} $\CC (\emptyset) = \emptyset$,\qquad
{\rm(2)} $\mathcal{X} \subseteq \CC(\mathcal{X} )$,\qquad
{\rm(3)} $\CC (\mathcal{X} \cup \mathcal{Y} ) = \CC (\mathcal{X} ) \cup \CC (\mathcal{Y})$,\qquad
{\rm(4)} $\CC (\CC ( \mathcal{X}) ) = \CC (\mathcal{X} )$
\\
hold for any subsets $\XX,\YY$ of $\modR$. 
In particular, it defines a topology on $\modR$: a subset $\XX$ of $\modR$ is closed if and only if $\CC(\XX)=\XX$, if and only if $\XX=\CC(\YY)$ for some subset $\YY$ of $\modR$.
\end{theorem}

\begin{proof}
It is straightforward to show the assertions (1) and (3), while Proposition \ref{A2}(4) implies that the assertion (2) holds.
Let us show the assertion (4).
Pick any module $N\in\CC(\CC(\XX))$.
Then there are a module $M\in\CC(\XX)$ and an integer $a>0$ such that $M^{\oplus a}\dega N^{\oplus a}$.
Hence there are a module $L\in\XX$ and an integer $b>0$ such that $L^{\oplus b}\dega M^{\oplus b}$.
Using Proposition \ref{A2}(3), we get degenerations $M^{\oplus ab}\dega N^{\oplus ab}$ and $L^{\oplus ab}\dega M^{\oplus ab}$.
Applying Theorem \ref{T}, we obtain a degeneration $L^{\oplus abc}\dega N^{\oplus abc}$ for some integer $c>0$.
This says that $N$ belongs to $\CC(\XX)$, which proves that $\CC(\CC(\XX))$ is contained in $\CC(\XX)$.
The opposite inclusion follows from (2), and we conclude that the equality $\CC(\CC(\XX))=\CC(\XX)$ holds.
\end{proof}

In the remainder of this paper, whenever we consider the set $\modR$ of isomorphism classes of finitely generated $R$-modules, we equip it with the topology defined in the above Theorem \ref{A4}.

We close this section by stating a corollary of Theorem \ref{A4}.

\begin{corollary}\label{4}
\begin{enumerate}[\rm(1)]
\item
For subsets $\XX,\YY$ of $\modR$ with $\XX\subseteq\YY$ it holds that $\CC(\XX)\subseteq\CC(\YY)$.
\item
Let $M,N\in\modR$ be such that $M^{\oplus n} \dega N^{\oplus n}$ for some $n>0$.
One then has $\CC(N) \subseteq \CC(M)$.
\item
The set $\CC(M)$ is an irreducible closed subset of $\modR$ for each $M\in\modR$.
\item
For every subset $\XX$ of $\modR$ there exists a decomposition $\CC(\XX) = \bigcup _{M \in \XX }  \CC (M)$ into irreducible closed subsets. 
\end{enumerate}
\end{corollary}

\begin{proof}
(1) As $\YY=\XX\cup\YY$, we have $\CC(\YY)=\CC(\XX\cup\YY)=\CC(\XX)\cup\CC(\YY)\supseteq\CC(\XX)$ by Theorem \ref{A4}(3).

(2) The assumption implies $N\in\CC(M)$.
Hence $\CC(N)\subseteq\CC(\CC(M))=\CC(M)$ by (1) and Theorem \ref{A4}(4).

(3) Theorem \ref{A4} implies that $\CC(M)$ is closed.
Assume $\CC(M)=\XX\cup\YY$ for some closed subsets $\XX,\YY$ of $\modR$.
As $M\in\CC(M)$ by Theorem \ref{A4}(2), we may assume $M\in\XX$.
Hence $\CC(M)$ is contained in $\CC(\XX)$, which coincides with $\XX$ as $\XX$ is closed.
Therefore $\CC(M)=\XX$, which shows that $\CC(M)$ is irreducible.

(4) We can directly verify that $\CC(\XX)$ is contained in $\bigcup _{M \in \XX }  \CC (M)$.
Using (1), we obtain the equality $\CC(\XX) = \bigcup _{M \in \XX }  \CC (M)$.
It follows from (3) that the sets $\CC(M)$ are irreducible closed subsets.
\end{proof}

\section{Irreducible components of $\modR$}\label{Irreducible}

Throughout this section, we assume that $(R, \m)$ is a CM complete local ring with coefficient field $k$. 
We denote by $\CMR$ the set of isomorphism classes of MCM $R$-modules, which is a subspace of the topological space $\modR$.
In what follows, for each subset $\XX$ of $\CMR$ 
we consider the restriction of $\CC(\XX)$ to $\CMR$, namely, we investigate the subset
$$
\cm(\XX):=\CC(\XX)\cap\CMR= \{ N \in \CMR \mid M^{\oplus n} \dega N^{\oplus n} \text{ for some }M\in\XX\text{ and }n>0 \}
$$
of $\CMR$.
We set $\cm(M)=\cm(\{M\})$ for $M\in \CMR$.
Here are some basic properties.

\begin{proposition}\label{B1}
\begin{enumerate}[\rm(1)]
\item
The assignment $\XX\mapsto\cm(\XX)$ induces a Kuratowski closure operator on $\CMR$.
In particular, a subset $\XX$ of $\CMR$ is closed if and only if $\cm(\XX)=\XX$, if and only if $\XX=\cm(\YY)$ for some subset $\YY$ of $\CMR$.
\item
For subsets $\XX,\YY$ of $\CMR$ with $\XX\subseteq\YY$ it holds that $\cm(\XX)\subseteq\cm(\YY)$.
\item
Let $M,N\in\CMR$ be such that $M^{\oplus n} \dega N^{\oplus n}$ for some $n>0$.
One then has $\cm(N) \subseteq \cm(M)$.
\item
The set $\cm(M)$ is an irreducible closed subset of $\CMR$ for each $M\in\CMR$.
\item
For every subset $\XX$ of $\CMR$ there exists a decomposition $\cm(\XX) = \bigcup _{M \in \XX }  \cm (M)$ into irreducible closed subsets. 
\end{enumerate}
\end{proposition}

\begin{proof}
Let $\XX,\YY$ be subsets of $\CMR$.
It is easy to observe from Theorem \ref{A4}(1)(2)(3) that $\cm(\emptyset)=\emptyset$, $\XX\subseteq\cm(\XX)$ and $\cm(\XX\cup\YY)=\cm(\XX)\cup\cm(\YY)$.
Since $\cm(\XX)\subseteq\CC(\XX)$, we have $\cm(\cm(\XX))\subseteq\cm(\CC(\XX))\subseteq\CC(\CC(\XX))=\CC(\XX)$, where the last equality follows from Theorem \ref{A4}(4).
Hence $\cm(\cm(\XX))\subseteq\CC(\XX)\cap\CMR=\cm(\XX)$, and therefore $\cm(\cm(\XX))=\cm(\XX)$.
Thus the first assertion of the proposition follows.
The remaining assertions of the proposition are shown along the same lines as in the proof of Corollary \ref{4}.
\end{proof}

For each integer $d>0$ we denote by $\X(d)$ the subset of $\CMR$ consisting of MCM modules of multiplicity $d$, that is,
$$
\X(d) = \{ M \in \CMR \mid \e(M) = d \}.
$$
Proposition \ref{A2}(6) guarantees that
\begin{equation}\label{3}
\XX\subseteq\X(d)\implies\cm(\XX)\subseteq \X(d).
\end{equation}
Hence $\X(d)$ is a variant of the module variety $\md$ defined (only) in the finite-dimensional case.

Recall that $R$ is said to have {\em finite CM representation type} if there are only a finite number of isomorphism classes of indecomposable MCM modules.
When this is the case, the topological space $\X(d)$ can be decomposed into finitely many irreducible closed subsets.

\begin{corollary}\label{B2}
Suppose that $R$ has finite CM representation type.
Then for every integer $d>0$ the topological space $\X(d)$ has a decomposition
$$
\X(d) = \bigcup _{i = 1} ^n \cm(M_i)
$$
into finitely many irreducible closed subsets, where $M_1,\dots,M_n$ are MCM $R$-modules of multiplicity $d$.
\end{corollary}

\begin{proof}
As $R$ has finite CM representation type, there exist finitely many MCM $R$-modules $M_1,\dots,M_n$ of multiplicity $d$ such that $\X(d)=\{M_1,\dots,M_n\}$.
It is easy to verify by Theorem \ref{A4}(2) and \eqref{3} that $\X(d)=\bigcup _{i = 1} ^n \cm(M_i)$.
By Proposition \ref{B1}(1) each $\cm(M_i)$ is an irreducible closed subset of $\X(d)$.
\end{proof}

We give two examples of computation of the topological space $\X(d)$.

\begin{example}\label{1}
\begin{enumerate}[(1)]
\item
Let $R=k[\![ x, y]\!]/(x^3 + y^2)$. 
Then the nonisomorphic indecomposable MCM $R$-modules are $R$ and $\m=(x,y)$, which implies $\X(4) = \{R ^{\oplus 2}, R \oplus \m, \m ^{\oplus 2} \}$. 
The MCM module $R ^{\oplus 2}$ does not degenerate to $R \oplus \m$ by \cite[Proposition 5.3]{Y11}, but the exact sequence $0\to\m\to R^{\oplus2}\to\m\to0$ yields a degeneration $R^{\oplus2}\dega\m^{\oplus2}$ by Proposition \ref{A2}(2), which gives rise to a degeneration $R ^{\oplus 2} \oplus R ^{\oplus 2} \dega \m ^{\oplus 2} \oplus R^{\oplus 2}$ by Proposition \ref{A2}(3)(4).
Hence $\X(4) = \cm (R ^{\oplus 2})$. 
\item
Let $R=k[\![ x, y, z]\!]/(x^3 + yz)$. 
Then the nonisomorphic indecomposable MCM $R$-modules are $R$, $I=(x,y)$ and $J=(x^2,y)$, which implies $\X(6)=\{R^{\oplus a}\oplus I^{\oplus b}\oplus J^{\oplus c}\mid a+b+c=3\}$.
It follows from \cite[Theorem 3.1]{HY13} that $M \dega N$ if and only if there is an exact sequence $0\to A\to M\to B\to0$ of MCM modules such that $N=A\oplus B$ for all $M, N \in \CMR$.
We obtain $\X(6) = \CC (R ^{\oplus 3}) \cup \CC(R^{\oplus 2} \oplus I ) \cup \CC(R^{\oplus 2} \oplus J )$. 
See \cite[Example 3.13]{HY13}.
\end{enumerate}
\end{example}

The main result of this section is the following theorem.

\begin{theorem}\label{main}
Let $k$ be an algebraically closed uncountable field of characteristic not two.
Let $R$ be either an odd-dimensional hypersurface of type $(\A_{\infty})$ or an arbitrary-dimensional hypersurface of type $(\D_{\infty})$ over $k$. 
Then $\X (d)$ can be represented by a finite union of irreducible closed subsets for any $d>0$.
\end{theorem}

In the case where $R$ is a finite-dimensional $k$-algebra, the module variety $\md$ can always be described as a finite union of irreducible closed subsets, since it is an (affine) algebraic set. 
The topological space $\X(d)$ (and hence $\CMR$) is not even noetherian in general. 

\begin{example}
Let $k$ be algebraically closed.
Let $R=k[\![x,y]\!]/(x^2)$.
Then there exists a descending chain
$$
\cm(R) \supsetneq \cm((x, y^2)) \supsetneq \cm((x, y^4)) \supsetneq \cdots \supsetneq \cm((x, y^{2n})) \supsetneq \cm((x, y^{2n +2})) \supsetneq \cdots
$$
of irreducible closed subsets; see \cite[Theorem 1.1]{HTY} and Proposition \ref{B1}(3)(4).
\end{example}

From now on, we show several results to give a proof of Theorem \ref{main}.
We begin with a lemma.

\begin{lemma}\label{5}
Let $M,N\in\CMR$.
For all $A\in\cm(M)$ and $B\in\cm(N)$ one has $A\oplus B\in\cm(M\oplus N)$.
\end{lemma}

\begin{proof}
There exist degenerations $M^{\oplus a}\dega A^{\oplus a}$ and $N^{\oplus b}\dega B^{\oplus b}$, where $a,b$ are positive integers.
Applying Proposition \ref{A2}(3) a couple of times, we get degenerations $M^{\oplus ab}\dega A^{\oplus ab}$ and $N^{\oplus ab}\dega B^{\oplus ab}$, and obtain a degeneration $(M\oplus N)^{\oplus ab}\dega(A\oplus B)^{\oplus ab}$.
Hence $A\oplus B\in\cm(M\oplus N)$.
\end{proof}

For a finitely generated $R$-module $M$, we denote by $\syz M$ the (first) {\em syzygy} of $M$, i.e., the image of the first differential map in the minimal free resolution of $M$.
The following proposition is a more precise version of Theorem \ref{main} in the case where the ring $R$ is either a $1$-dimensional hypersurface of type $(\A_{\infty})$ or a $2$-dimensional hypersurface of type $(\D_{\infty})$.

\begin{proposition}\label{B4}
Let $d>0$ be an integer.
\begin{enumerate}[\rm(1)]
\item
Let $R=k[\![x,y]\!]/(x^2)$. 
Then $\X(d) = \begin{cases}
\cm(R^{\oplus \frac{d-1}{2}} \oplus(x)) & \text{if $d$ is odd}, \\
\cm(R^{\oplus \frac{d}{2}} ) & \text{if $d$ is even}.
\end{cases}$
\item
Let $R= k[\![x, y, z]\!]/(x^2 y + z^2)$.  
Then $\X(d)=\begin{cases}\cm(R^{\oplus\frac{d}{2}}) & \text{if $d$ is even},\\
\emptyset & \text{if $d$ is odd}.\end{cases}$
\end{enumerate}
\end{proposition}

\begin{proof}
(1) The nonisomorphic indecomposable MCM $R$-modules are
$$
R,\,
(x)\cong R/(x),\,
I_n = (x,y^n)\ (n \geq 1);
$$
see \cite[Proposition 4.1]{BGS87}.
There are isomorphisms $(x)\cong\syz(x)$ and $I_n\cong\syz I_n$, which gives exact sequences $0\to(x)\to R\to(x)\to0$ and $0\to I_n\to R^{\oplus2}\to I_n\to0$.
It is observed from Proposition \ref{A2}(2) that $R,(x)^{\oplus2},I_n$ belong to $\cm(R)$.
Pick any $X\in\X(d)$ and write
$$
X=R^{\oplus a}\oplus(x)^{\oplus b}\oplus I_{l_1}^{\oplus c_1}\oplus\cdots\oplus I_{l_m}^{\oplus c_m}.
$$
As $\e(R)=\e(I_n)=2$ and $\e((x))=\e(R/(x))=1$, we have $d=\e(X)=2(a+c_1+\cdots+c_m)+b$.
Assume that $d$ is odd (resp. even).
Then $b=2r-1$ (resp. $b=2r$) for some $r\ge1$, and $d=2(a+c_1+\cdots+c_m+r)-1$ (resp. $d=2(a+c_1+\cdots+c_m+r)$).
Applying Lemma \ref{5}, we see that $X$ belongs to $\cm(R^{\oplus a+c_1+\cdots+c_m+r}\oplus (x))$ (resp. $\cm(R^{\oplus a+c_1+\cdots+c_m+r})$).
Therefore the left-hand side of the equality in the assertion is contained in the right-hand side.
The opposite inclusion is easily seen by using \eqref{3}.

(2) The nonisomorphic indecomposable MCM $R$-modules are
$$
R,\,
I=(x,z),\,
J=(y,z),\,
M_n =\Coker \left( \begin{smallmatrix} x&y^n&z&0\\0&-x&0&z\\ -z& 0& xy&y^{n+1}\\ 0& -z& 0& -xy\end{smallmatrix}\right),\,
N_n = \Coker \left( \begin{smallmatrix} x&y^n&z&0\\0&-xy&0&z\\ -z& 0& xy&y^n\\ 0& -z& 0& -x\end{smallmatrix}\right)\,(n\ge1);
$$
see \cite[Proposition 14.19]{LW12}. 
Note that $I \cong \Omega I$, $J \cong \Omega J$, $M_n \cong \Omega M_n$ and $N_n \cong \Omega N_n$.
We make a similar argument as in the proof of (1).
By Proposition \ref{A2}(2) we have
$$
I,J\in\cm(R),\quad
M_n,N_n\in\CC(R ^{\oplus 2}),\quad
\e(R)=\e(I)=\e(J)=2,\quad
\e(M_n)=\e(N_n)=4.
$$
(Hence every MCM module has even multiplicity.)
Take any module $X\in\X(d)$, and write
$$
\textstyle
X=R^{\oplus a}\oplus I^{\oplus b}\oplus J^{\oplus c}\oplus(\bigoplus_iM_{p_i}^{\oplus d_i})\oplus(\bigoplus_jN_{q_j}^{\oplus e_j}).
$$
Then $d=\e(X)=2(a+b+c+2\sum_id_i+2\sum_je_j)$, and $X\in\cm(R^{\oplus a+b+c+2\sum_id_i+2\sum_je_j})$ by Lemma \ref{5}.
Finally, by using \eqref{3}, the assertion follows.
\end{proof}

The following proposition is nothing but Theorem \ref{main} in the case where the ring $R$ is a $1$-dimensional hypersurface of type $(\D_{\infty})$.
The proof uses matrix factorizations. 
We refer to \cite[$\S$ 7]{Y} for the details. 

\begin{proposition}\label{B5}
Let $R =k[\![x,y]\!]/(x^2y)$. 
Then for each integer $d>0$ one has $\X(d)=\cm(\XX)$, where $\XX$ is the set of isomorphism classes of modules of the form
$$
R^{\oplus l_1}\oplus(xy)^{\oplus l_2}\oplus(x)^{\oplus l_3}\oplus(y)^{\oplus l_4}\oplus(x^2)^{\oplus l_5}\oplus(\Coker \left( \begin{smallmatrix} x & y \\ 0& -x\end{smallmatrix}\right))^{\oplus l_6}\oplus(\Coker \left( \begin{smallmatrix} xy & y^2 \\ 0& -xy\end{smallmatrix}\right))^{\oplus l_7},
$$
where $d=3l_1 + l_2 + 2l_3 + 2l_4+ l_5 + 2l_6+4l_7$.
In particular, the topological space $\X(d)$ is a finite union of irreducible closed subsets.
\end{proposition}

\begin{proof}
We prove the proposition similarly as in the proof of Proposition \ref{B4}.
By \cite[Proposition 4.2]{BGS87}, the nonisomorphic indecomposable MCM $R$-module are
\begin{align*}
&R,\,
(xy)\cong R/(x),\,
(x)\cong R/(xy),\,
(y)\cong R/(x^2),\,
(x^2)\cong R/(y),\\
&M_n ^{+} = \Coker \left( \begin{smallmatrix} x & y^n \\ 0& -x\end{smallmatrix}\right),\,
M_n ^{-} = \Coker \left( \begin{smallmatrix} xy & y^{n+1} \\ 0& -xy\end{smallmatrix}\right),\,
I_n ^{+} = (x,y^n),\,
I_n ^{-} = (xy,y^n)\ (n\ge1).
\end{align*}
It is seen that $\Omega M_n ^{\pm} \cong M_n ^{\mp}$ and $\Omega I_n ^{\pm} \cong I_n ^{\mp}$.
Moreover, we have
\begin{align*}
&\e(R/(x)) = \e(R/(y)) = 1,\quad
\e(R/(xy))=\e(R/(x^2)) = \e(M_n ^+ ) = 2,\\
&\e(R)= \e(I_n ^+ ) = \e(I_n ^-) = 3,\quad
\e(M_n ^-) = 4.
\end{align*}
In fact, for instance, since there exist a short exact sequence $0 \to R/(x) \to M_n ^+ \to R/(x) \to 0$ (see \cite[Proposition 2.1]{AIT12}), we have $\e (M_n ^+) = \e(R/(x))+\e(R/(x))=2$.
Now, let $n\ge0$.
Set $V=k[t]_{(t)}$ and consider the $R \otimes _kV$-modules $Q_{n} ^{+},Q_{n} ^{-},P_{n} ^{+},P_{n} ^{-}$ whose presentation matrices are
\begin{equation}\label{6}
\left( \begin{smallmatrix} x + ty^{n+1}& y^{n+2} \\ t^2y^n & -x + ty^{n+1} \end{smallmatrix}\right),\, 
\left( \begin{smallmatrix} xy - ty^{n+2}& y^{n+3} \\ t^2 y^{n+1} & -xy - ty^{n+2} \end{smallmatrix}\right),\, 
\left( \begin{smallmatrix} x + ty^{n}& t^2y^{n} \\ y^{n+1} & -xy + ty^{n+1} \end{smallmatrix}\right),\, 
\left( \begin{smallmatrix} xy - ty^{n+1}& t^2 y^{n} \\ y^{n+1} & -x - ty^{n} \end{smallmatrix}\right)
\end{equation}
respectively. 
Then $Q_{n} ^{\pm}$ and $P_{n} ^{\pm}$ give the degenerations
\begin{equation}\label{100}
M_n ^{\pm} \dega M_{n+2}^{\pm},\quad
I_n ^{\mp} \dega I_{n+1}^{\pm}\quad(n\ge0)
\end{equation}
respectively. 
Indeed,
$$
(\left( \begin{smallmatrix} x + ty^{n+1}& y^{n+2} \\ t^2y^n & -x + ty^{n+1} \end{smallmatrix}\right),\, 
\left( \begin{smallmatrix} xy - ty^{n+2}& y^{n+3} \\ t^2 y^{n+1} & -xy - ty^{n+2} \end{smallmatrix}\right) ),\qquad
(\left( \begin{smallmatrix} x + ty^{n}& t^2y^{n} \\ y^{n+1} & -xy + ty^{n+1} \end{smallmatrix}\right),\, 
\left( \begin{smallmatrix} xy - ty^{n+1}& t^2 y^{n} \\ y^{n+1} & -x - ty^{n} \end{smallmatrix}\right))
$$
are matrix factorizations of $x^2y$. 
Thus $Q_{n} ^{\pm}$ and $P_{n} ^{\pm}$ are MCM $R \otimes _k V$-modules, which are $V$-flat. 
We also have morphisms of matrix factorizations
\begin{align*}
(\left( \begin{smallmatrix} 0 &-1 \\ t^2 & -ty \end{smallmatrix}\right),\left( \begin{smallmatrix} 0 &1 \\ -t^2 & -ty \end{smallmatrix}\right)  ) &: (\left( \begin{smallmatrix} x + ty^{n+1}& y^{n+2} \\ t^2y^n & -x + ty^{n+1} \end{smallmatrix}\right),\, 
\left( \begin{smallmatrix} xy - ty^{n+2}& y^{n+3} \\ t^2 y^{n+1} & -xy - ty^{n+2} \end{smallmatrix}\right) \to (\left( \begin{smallmatrix} x &y^n \\ 0 & -x \end{smallmatrix}\right),\left( \begin{smallmatrix} xy &y^{n+1} \\ 0 & -xy \end{smallmatrix}\right)),
\\
(\left( \begin{smallmatrix} 1 &0 \\ -ty & t^2 \end{smallmatrix}\right),\left( \begin{smallmatrix} 1 &0 \\ t & t^2 \end{smallmatrix}\right)  ) &: (\left( \begin{smallmatrix} x + ty^{n}& t^2y^{n} \\ y^{n+1} & -xy + ty^{n+1} \end{smallmatrix}\right),\, 
\left( \begin{smallmatrix} xy - ty^{n+1}& t^2 y^{n} \\ y^{n+1} & -x - ty^{n} \end{smallmatrix}\right)) \to (\left( \begin{smallmatrix} x &y^n \\ 0 & -xy \end{smallmatrix}\right),\left( \begin{smallmatrix} xy &y^{n} \\ 0 & -x \end{smallmatrix}\right)).
\end{align*}
Since these are isomorphisms if $t$ is invertible, there are $R \otimes_k K$-isomorphisms $(Q_{n} ^{\pm})_t \cong M_{n}^{\pm} \otimes _k K$ and $(P_{n} ^{\pm})_t \cong I_{n}^{\pm} \otimes _k K$. 
It is clear that $Q_{n} ^{\pm} \otimes _k V/tV \cong M_{n+2} ^{\pm}$ and $P_{n} ^{\pm} \otimes _k V/tV \cong I_{n+1} ^{\pm}$. 
Thus we obtain the degenerations \eqref{100}.
Hence each indecomposable MCM module belongs to
$$
\cm(\{R,\,(xy),\,(x),\,(y),\,(x^2),\,M_1^+,\,M_1^-,\,M_0^+,\,M_0 ^-,\,I_0^+,\,I_0^- \}).
$$
Note here that $M_0 ^{+} \cong R/(x^2)$, $M_0 ^{-} \cong R \oplus R/(y)$ and $I_0 ^{+}= I_0 ^{-}= R$. 
For each $X\in\X(d)$ there exist integers $l_1,\dots,l_7$ with $d=3l_1 + l_2 + 2l_3 + 2l_4+ l_5 + 2l_6+4l_7$ such that
$$
X\in\cm(R^{\oplus l_1}\oplus(xy)^{\oplus l_2}\oplus(x)^{\oplus l_3}\oplus(y)^{\oplus l_4}\oplus(x^2)^{\oplus l_5}\oplus(M_1^+)^{\oplus l_6}\oplus(M_1^-)^{\oplus l_7}).
$$
Now the proof of the proposition is completed.
\end{proof}

In the proof of Proposition \ref{B5}, we construct $R \otimes_k V$-modules $Q_n^{\pm}$ and $P_n ^{\pm}$ concretely. 
Using them, we can also show the case where the ring $R$ is a $3$-dimensional hypersurface of type $(\D_{\infty})$.
Let $R=S/(f)$ be a hypersurface. 
Kn\"{o}rrer's periodicity theorem \cite[\S 12]{Y04} gives rise to the functor $(-)^{\sharp \sharp} :\CM (R) \to \CM (\Rss )$, where $\Rss = S[\![u, v]\!]/(f + u^2 + v^2)$. 
We call this functor {\em Kn\"{o}rrer's periodicity functor}. 

\begin{proposition}\label{B7}
Let $R^{\sharp \sharp} =k[\![x,y, u, v]\!]/(x^2y + u ^2 + v^2)$. 
Then for each integer $d>0$ one has $\X(d)=\cm(\XX)$, where $\XX$ is the set of isomorphism classes of modules of the form
\begin{multline*}
(R^{\sharp \sharp})^{\oplus l_1}\oplus \{ (xy)^{\sharp \sharp}\} ^{\oplus l_2}\oplus \{ (x)^{\sharp \sharp}\} ^{\oplus l_3}\oplus \{ (y)^{\sharp \sharp}\} ^{\oplus l_4}\oplus \{ (x^2)^{\sharp \sharp}\} ^{\oplus l_5}\\
\oplus \{ (\Coker \left( \begin{smallmatrix} x & y \\ 0& -x\end{smallmatrix}\right))^{\sharp \sharp}\} ^{\oplus l_6}\oplus \{ (\Coker \left( \begin{smallmatrix} xy & y^2 \\ 0& -xy\end{smallmatrix}\right))^{\sharp \sharp}\} ^{\oplus l_7}
\end{multline*}
with $d=2l_1 + 2l_2 + 2l_3 + 2l_4+ 2l_5 + 4l_6+4l_7$, where $(xy)$, $(x)$, $(y)$, $(x^2)$, $\Coker \left( \begin{smallmatrix} x & y \\ 0& -x\end{smallmatrix}\right)$ and $\Coker \left( \begin{smallmatrix} xy & y^2 \\ 0& -xy\end{smallmatrix}\right)$ are the MCM $R$-modules given in Proposition \ref{B5}.
In particular, the topological space $\X(d)$ is a finite union of irreducible closed subsets. 
\end{proposition}

\begin{proof}
As shown in the proof of \cite[Theorem 12.10]{Y}, each nonfree indecomposable MCM $R^{\sharp \sharp}$-module has the form $M^{\sharp \sharp}$ for some nonfree  indecomposable MCM module $M$ over $R = k[\![x,y ]\!]/(x^2y )$.
With the notation of Proposition \ref{B5}, the modules $\Rss$, $(xy)^{\sharp \sharp}$, $(x)^{\sharp \sharp}$, $(y)^{\sharp \sharp}$, $(x^2)^{\sharp \sharp}$, $(M_n ^+)^{\sharp \sharp}$, $(M_n ^-)^{\sharp \sharp}$, $(I_n ^+)^{\sharp \sharp},(I_n ^-)^{\sharp \sharp}$ for $n \geq1$ are nonisomorphic indecomposable MCM $\Rss$-modules.  
We claim that 
$$
 ( M_n ^{\pm})^{\sharp \sharp}  \dega (M_{n+2}^{\pm})^{\sharp \sharp},\quad
 ( I_n ^{\mp})^{\sharp \sharp}  \dega ( I_{n+1}^{\pm})^{\sharp \sharp} \quad (n\ge0)
$$
respectively.
In fact, let $(\Phi, \Psi)$ be a matrix factorization of $Q_{n} ^{+}$; see (\ref{6}). 
Consider the pair of matrices
$$
(\Phi ^{\sharp \sharp}   , \Psi ^{\sharp \sharp}  )= (\small \left( \begin{smallmatrix}\Phi & u + \sqrt{-1}v \\ -u+\sqrt{-1}v& \Psi \end{smallmatrix}\right), \left( \begin{smallmatrix}\Psi & -u - \sqrt{-1}v \\ u-\sqrt{-1}v& \Phi \end{smallmatrix}\right) ).
$$
We see that $ (Q_{n} ^{+})^{\sharp \sharp}  = \Coker ( \Phi ^{\sharp \sharp} )$ (resp. $(Q_{n} ^{-})^{\sharp \sharp}  = \Coker (\Psi ^{\sharp \sharp})$) gives a degeneration $( M_n ^{+ })^{\sharp \sharp}  \dega (M_{n+2}^{+ })^{\sharp \sharp} $ (resp. $( M_n ^{- })^{\sharp \sharp}  \dega (M_{n+2}^{- })^{\sharp \sharp}$).
As $(\Phi ^{\sharp \sharp}   , \Psi ^{\sharp \sharp}  )$ is a matrix factorization of $x^2y + u^2 + v^2$, the $R^{\sharp \sharp} \otimes _k V$-modules $ (Q_{n} ^{+})^{\sharp \sharp}$ and $(Q_{n} ^{-})^{\sharp \sharp}$ are MCM, whence $V$-flat. 
Set $\alpha = \left( \begin{smallmatrix}  0 & -1\\ t^2&-ty\end{smallmatrix}\right)$ and $\beta = \left( \begin{smallmatrix}  0 & 1\\ -t^2&-ty \end{smallmatrix}\right)$. 
We have $\left( \begin{smallmatrix} \alpha & 0\\ 0&\beta\end{smallmatrix}\right) \Phi^{\sharp \sharp} = \varphi_n ^{\sharp \sharp}\left( \begin{smallmatrix}  \beta & 0\\ 0&\alpha \end{smallmatrix}\right)$ and $\left( \begin{smallmatrix}  \beta & 0\\ 0&\alpha\end{smallmatrix}\right) \Psi^{\sharp \sharp} = \psi_n ^{\sharp \sharp}\left( \begin{smallmatrix}  \alpha & 0\\ 0&\beta \end{smallmatrix}\right)$, where $\varphi _n =\left( \begin{smallmatrix}  x & y^n\\ 0&-x\end{smallmatrix}\right)$ and $\psi _n =\left( \begin{smallmatrix}  xy & y^{n+1}\\ 0&-xy\end{smallmatrix}\right)$. 
Hence $(\left( \begin{smallmatrix} \alpha & 0\\ 0&\beta\end{smallmatrix}\right), \left( \begin{smallmatrix}  \beta & 0\\ 0&\alpha\end{smallmatrix}\right) )$ gives a morphism of matrix factorizations $(\Phi ^{\sharp \sharp}   , \Psi ^{\sharp \sharp}  ) \to (\varphi _n ^{\sharp \sharp}, \psi _n ^{\sharp \sharp})$. 
Since $\alpha \otimes_k K$ and $\beta \otimes_k K$ are invertible, $\left( \begin{smallmatrix}  \alpha & 0\\ 0&\beta\end{smallmatrix}\right)\otimes_k K$ and $\left( \begin{smallmatrix}  \beta & 0\\ 0&\alpha\end{smallmatrix}\right)\otimes_k K$ are also invertible. 
Thus $(Q_{n} ^{\pm })^{\sharp \sharp}  \otimes_V K \cong (M_{n}^{\pm})^{\sharp \sharp} \otimes_k K$. 
Clearly  $(Q_{n} ^{\pm })^{\sharp \sharp}  \otimes_V V/tV \cong (M_{n+2}^{\pm})^{\sharp \sharp}$, so that we obtain a degeneration  $ ( M_n ^{\pm})^{\sharp \sharp}  \dega (M_{n+2}^{\pm})^{\sharp \sharp}$. 
Similarly, we can show $(I_n ^{\mp})^{\sharp \sharp}  \dega ( I_{n+1}^{\pm})^{\sharp \sharp}$.
Thus the claim follows.
Note that there are equalities
$$
\e( R^{\sharp \sharp}) = \e ( (xy)^{\sharp \sharp})= \e(  (x)^{\sharp \sharp}) = \e ((y)^{\sharp \sharp}) = \e((x^2)^{\sharp \sharp} ) = 2,\qquad
\e( ( M_n ^{\pm})^{\sharp \sharp} )=\e( (I_n ^{\pm})^{\sharp \sharp})=4. 
$$
Similar arguments as in the proof of Proposition \ref{B5} yield the assertion. 
\end{proof}

Now we are ready to prove Theorem \ref{main}. 

\begin{proof}[\bf Proof of Theorem \ref{main}]
Let $R=S/(f)$ be a hypersurface as in Propositions \ref{B4} and \ref{B7}, and assume that the base field $k$ is algebraically closed and has characteristic not two.
The proofs of those propositions show that there exist a finite number of MCM $R$-modules $G_1,\dots,G_n$ (depending only on $R$) such that all the isomorphism classes of indecomposable MCM $R$-modules belong to $\cm(\{G_1,\dots,G_n\})$.
It follows from \cite[Proposition 5.3]{HTY} that all the isomorphism classes of indecomposable MCM $\Rss$-modules are in $\cm(\{G_1^{\sharp\sharp},\dots,G_n^{\sharp\sharp}\})$, where $\Rss=S[\![u,v]\!]/(f+u^2+v^2)$ is a hypersurface and $G_i^{\sharp\sharp}$ stands for the image of $G_i$ by the Kn\"orrer periodicity functor.

The case where the ring $R$ is a 1-dimensional hypersurface of type $(D_{\infty})$ follows from Proposition \ref{B5}.
Now let $R$ be one of the other hypersurfaces in the theorem.
Iterating the above argument, we find MCM $R$-modules $G_1,\dots,G_n$ such that all the isomorphism classes of indecomposable MCM $R$-modules belong to $\XX:=\cm(\{G_1,\dots,G_n\})$.
Let $d>0$ be an integer, and take $M\in\X(d)$.
Then there exist indecomposable MCM $R$-modules $M_1,\dots,M_s$ and integers $e_1,\dots,e_s$ such that $M=M_1^{\oplus e_1}\oplus\cdots\oplus M_s^{\oplus e_s}$.
Each $M_i$ is in $\XX$, so $M_i\in\cm(G_{l_i})$ for some $l_i$.
Then $M\in\cm(G_{l_1}^{\oplus e_1}\oplus\cdots\oplus G_{\l_s}^{\oplus e_s})$ by Lemma \ref{5}.
Hence we obtain
$$
\X(d)=\cm(\{G_1^{\oplus a_1}\oplus\cdots\oplus G_n^{\oplus a_n}\mid a_1,\dots,a_n\ge0\text{ with }a_1\e(G_1)+\cdots+a_n\e(G_n)=d\}).
$$
This completes the proof of the theorem.
\end{proof}

In view of our Theorem \ref{main}, it is quite natural to ask what happens for even-dimensional hypersurfaces of type $(\A_\infty)$.
We end this section by stating a remark on this.

\begin{remark}
Let $R = k[\![ x, y, z]\!]/(xy)$, that is, the $2$-dimensional hypersurface of type $(\A_{\infty})$, and assume that $k$ is algebraically closed.
The nonisomorphic indecomposable MCM $R$-modules are
$$
R,\,
(x),\,
(y),\,
(x, z^n),\,
(y, z^n)\ (n \geq 1).
$$  
In \cite[Theorem 1.1]{HTY} it is shown that for all $a<b$, the MCM module $(x, z^a)$ (resp. $(y, z^a)$) does not degenerate to $(x, z^b)$ (resp. $(y,z^b)$). 
Hence we guess that $N \not\in \CC(M)$ for all indecomposable MCM $R$-modules $M$ and $N$ with $M \not \cong N$, so that $\X(d)$ cannot be described as a finite union of $\cm(M)$. 

By the way, it may occur that $N \in \CC(M)$ even if $M$ does not degenerate to $N$.
For example, let $R = k[\![ x, y]\!]/(x^3 + y^2)$. 
Then the exact sequence $0 \to \m \to R^{\oplus 2} \to \m \to 0$ shows that $R^{\oplus 2} \dega \m ^{\oplus 2}$ by Proposition \ref{A2}(2).
However $R\not\dega \m$ by \cite[Proposition 3.3]{Y11}. 
\end{remark}

\ifx\undefined\bysame 
\newcommand{\bysame}{\leavevmode\hbox to3em{\hrulefill}\,} 
\fi

\end{document}